\documentclass[11pt]{amsart}
\usepackage[margin=1in]{geometry}
\usepackage{amsmath,amsfonts,amsthm,amssymb,bbm}
\usepackage{graphicx,color,dsfont}
\usepackage{amsmath,amssymb,amsthm}
\usepackage[english]{babel}
\usepackage{graphicx}
\usepackage{mathrsfs}
\usepackage{indentfirst}
\usepackage{pstricks}
\usepackage[format=hang, font=small]{caption}
\usepackage{wrapfig}
\usepackage{booktabs}
\usepackage{color}
\usepackage{hyperref} 
\usepackage{cases} 
\usepackage[T1,OT1]{fontenc}

\usepackage[utf8]{inputenc}
\usepackage{mathtools}

\newtheorem{theorem}{Theorem}
\newtheorem{lemma}{Lemma}
\newtheorem{proposition}{Proposition}

\newtheorem{definition}{Definition}
\newtheorem{corollary}{Corollary}

\newtheorem{claim}{Claim}

\newtheorem{thm}{Theorem}[section]
 \newtheorem{cor}[thm]{Corollary}
 \newtheorem{lem}[thm]{Lemma}
 \newtheorem{prop}[thm]{Proposition}
 \theoremstyle{definition}
 
 \theoremstyle{remark}
 \newtheorem{rem}[thm]{Remark}
 
 \numberwithin{equation}{section}
\theoremstyle{remark}
\newtheorem{remark}[thm]{Remark}

\newcommand{\vertiii}[1]{{\left\vert\kern-0.25ex\left\vert\kern-0.25ex\left\vert #1
    \right\vert\kern-0.25ex\right\vert\kern-0.25ex\right\vert}}

\newcommand{\R}{{\mathbb R}}
\newcommand{\eu}{{\rm e}}
\newcommand{\C}{{\mathbb C}}

\DeclareMathOperator*{\supess}{sup\, ess}

\begin{document}

\selectlanguage{english}

\newcommand{\abs}[1]{\lvert#1\rvert}
\newcommand{\f}[2]{\frac{#1}{#2}}
\newcommand{\dpr}[2]{\langle #1,#2 \rangle}
\newcommand{\jnab}{\langle\nabla\rangle}
\newcommand{\jxi}{\langle\xi\rangle}
\newcommand{\jeta}{\langle\eta\rangle}
\newcommand{\tvp}{\tilde{\vp}}
\newcommand{\tB}{\tilde{B}}

\newcommand{\ppr}{p'}
\newcommand{\pnf}{+\infty}
\newcommand{\mnf}{-\infty}
\newcommand{\nf}{\infty}
\newcommand{\oop}{\frac{1}{p}}
\newcommand{\ooppr}{\frac{1}{p'}}
\newcommand{\ooq}{\frac{1}{q}}
\newcommand{\ooqpr}{\frac{1}{q'}}
\newcommand{\oor}{\frac{1}{r}}
\newcommand{\oorpr}{\frac{1}{r'}}

\newcommand{\al}{\alpha}
\newcommand{\be}{\beta}
\newcommand{\wh}[1]{\widehat{#1}}
\newcommand{\ga}{\gamma}
\newcommand{\Ga}{\Gamma}
\newcommand{\de}{\delta}
\newcommand{\ben}{\beta_n}
\newcommand{\De}{\Delta}
\newcommand{\ve}{\varepsilon}
\newcommand{\ze}{\zeta}
\newcommand{\Th}{\Theta}
\newcommand{\ka}{\kappa}
\newcommand{\la}{\lambda}
\newcommand{\laj}{\lambda_j}
\newcommand{\lak}{\lambda_k}
\newcommand{\La}{\Lambda}
\newcommand{\si}{\sigma}
\newcommand{\Si}{\Sigma}
\newcommand{\vp}{\varphi}
\newcommand{\om}{\omega}
\newcommand{\Om}{\Omega}

\newcommand{\ro}{{\mathbf R}}
\newcommand{\rn}{{\mathbf R}^n}
\newcommand{\rd}{{\mathbf R}^d}
\newcommand{\rmm}{{\mathbf R}^m}
\newcommand{\rone}{\mathbf R}
\newcommand{\rtwo}{\mathbf R^2}
\newcommand{\rthree}{\mathbf R^3}
\newcommand{\rfour}{\mathbf R^4}
\newcommand{\ronen}{{\mathbf R}^{n+1}}
\newcommand{\ku}{\mathbf u}
\newcommand{\kw}{\mathbf w}
\newcommand{\kf}{\mathbf f}
\newcommand{\kz}{\mathbf z}

\newcommand{\tn}{\mathbf T^n}
\newcommand{\tone}{\mathbf T^1}
\newcommand{\ttwo}{\mathbf T^2}
\newcommand{\tthree}{\mathbf T^3}
\newcommand{\tfour}{\mathbf T^4}

\newcommand{\zn}{\mathbf Z^n}
\newcommand{\zp}{\mathbf Z^+}
\newcommand{\zone}{\mathbf Z^1}
\newcommand{\zz}{\mathbf Z}
\newcommand{\ztwo}{\mathbf Z^2}
\newcommand{\zthree}{\mathbf Z^3}
\newcommand{\zfour}{\mathbf Z^4}

\newcommand{\hn}{\mathbf H^n}
\newcommand{\hone}{\mathbf H^1}
\newcommand{\htwo}{\mathbf H^2}
\newcommand{\hthree}{\mathbf H^3}
\newcommand{\hfour}{\mathbf H^4}

\newcommand{\cone}{\mathbf C^1}
\newcommand{\ctwo}{\mathbf C^2}
\newcommand{\cthree}{\mathbf C^3}
\newcommand{\cfour}{\mathbf C^4}

\newcommand{\sn}{\mathbf S^{n-1}}
\newcommand{\sone}{\mathbf S^1}
\newcommand{\stwo}{\mathbf S^2}
\newcommand{\sthree}{\mathbf S^3}
\newcommand{\sfour}{\mathbf S^4}

\newcommand{\lp}{L^{p}}
\newcommand{\lppr}{L^{p'}}
\newcommand{\lqq}{L^{q}}
\newcommand{\lr}{L^{r}}
\newcommand{\echi}{(1-\chi(x/M))}
\newcommand{\chip}{\chi'(x/M)}

\newcommand{\wlp}{L^{p,\infty}}
\newcommand{\wlq}{L^{q,\infty}}
\newcommand{\wlr}{L^{r,\infty}}
\newcommand{\wlo}{L^{1,\infty}}

\newcommand{\lprn}{L^{p}(\rn)}
\newcommand{\lptn}{L^{p}(\tn)}
\newcommand{\lpzn}{L^{p}(\zn)}
\newcommand{\lpcn}{L^{p}(\cn)}
\newcommand{\lphn}{L^{p}(\cn)}

\newcommand{\lprone}{L^{p}(\rone)}
\newcommand{\lptone}{L^{p}(\tone)}
\newcommand{\lpzone}{L^{p}(\zone)}
\newcommand{\lpcone}{L^{p}(\cone)}
\newcommand{\lphone}{L^{p}(\hone)}

\newcommand{\lqrn}{L^{q}(\rn)}
\newcommand{\lqtn}{L^{q}(\tn)}
\newcommand{\lqzn}{L^{q}(\zn)}
\newcommand{\lqcn}{L^{q}(\cn)}
\newcommand{\lqhn}{L^{q}(\hn)}

\newcommand{\lo}{L^{1}}
\newcommand{\lt}{L^{2}}
\newcommand{\li}{L^{\infty}}

\newcommand{\co}{C^{1}}
\newcommand{\ci}{C^{\infty}}
\newcommand{\coi}{C_0^{\infty}}

\newcommand{\ca}{\mathcal A}
\newcommand{\cs}{\mathcal S}
\newcommand{\cm}{\mathcal M}
\newcommand{\cf}{\mathcal F}
\newcommand{\cb}{\mathcal B}
\newcommand{\ce}{\mathcal E}
\newcommand{\cd}{\mathcal D}
\newcommand{\cn}{\mathcal N}
\newcommand{\cz}{\mathcal Z}
\newcommand{\crr}{\mathbf R}
\newcommand{\cc}{\mathcal C}
\newcommand{\ch}{\mathcal H}
\newcommand{\cq}{\mathcal Q}
\newcommand{\cp}{\mathcal P}
\newcommand{\cx}{\mathcal X}

\newcommand{\pv}{\textup{p.v.}\,}
\newcommand{\loc}{\textup{loc}}
\newcommand{\intl}{\int\limits}
\newcommand{\iintl}{\iint\limits}
\newcommand{\dint}{\displaystyle\int}
\newcommand{\diint}{\displaystyle\iint}
\newcommand{\dintl}{\displaystyle\intl}
\newcommand{\diintl}{\displaystyle\iintl}
\newcommand{\liml}{\lim\limits}
\newcommand{\suml}{\sum\limits}
\newcommand{\ltwo}{L^{2}}
\newcommand{\supl}{\sup\limits}
\newcommand{\df}{\displaystyle\frac}
\newcommand{\p}{\partial}
\newcommand{\Ar}{\textup{Arg}}
\newcommand{\abssigk}{\widehat{|\si_k|}}
\newcommand{\ed}{(1-\p_x^2)^{-1}}
\newcommand{\tT}{\tilde{T}}
\newcommand{\tV}{\tilde{V}}
\newcommand{\wt}{\widetilde}
\newcommand{\Qvi}{Q_{\nu,i}}
\newcommand{\sjv}{a_{j,\nu}}
\newcommand{\sj}{a_j}
\newcommand{\pvs}{P_\nu^s}
\newcommand{\pva}{P_1^s}
\newcommand{\cjk}{c_{j,k}^{m,s}}
\newcommand{\Bjsnu}{B_{j-s,\nu}}
\newcommand{\Bjs}{B_{j-s}}
\newcommand{\Ly}{L_i^y}
\newcommand{\dd}[1]{\f{\partial}{\partial #1}}
\newcommand{\czz}{Calder\'on-Zygmund}

\newcommand{\lbl}{\label}
\newcommand{\beq}{\begin{equation}}
\newcommand{\eeq}{\end{equation}}
\newcommand{\beqna}{\begin{eqnarray*}}
\newcommand{\eeqna}{\end{eqnarray*}}
\newcommand{\beqn}{\begin{equation*}}
\newcommand{\eeqn}{\end{equation*}}
\newcommand{\bp}{\begin{proof}}
\newcommand{\ep}{\end{proof}}
\newcommand{\bprop}{\begin{proposition}}
\newcommand{\eprop}{\end{proposition}}
\newcommand{\bt}{\begin{theorem}}
\newcommand{\et}{\end{theorem}}
\newcommand{\bex}{\begin{Example}}
\newcommand{\eex}{\end{Example}}
\newcommand{\bc}{\begin{corollary}}
\newcommand{\ec}{\end{corollary}}
\newcommand{\bcl}{\begin{claim}}
\newcommand{\ecl}{\end{claim}}
\newcommand{\bl}{\begin{lemma}}
\newcommand{\el}{\end{lemma}}

\newcommand{\bib}[4]{\bibitem{#1}{\sc#2: }{\it#3. }{#4.}}

\title[Fractional Schr\"odinger-Poisson-Slater system in one dimension]
 {Fractional Schr\"odinger-Poisson-Slater system in one dimension}

\author[A. R. Giammetta]{Anna Rita Giammetta}
   \address{Department of Mathematics, University of Pisa,
   Italy}
   \email{giammetta@mail.dm.unipi.it}
   


\date{\today}

\begin{abstract}
In this paper we study local and global well-posedness of the following Cauchy problem: 

\begin{subnumcases}
{\,}
i\partial_t\Psi+\frac{1}{2}\Delta_{x}\Psi = A_0\Psi +\alpha |\Psi|^{\gamma-1}\Psi \,\,\,\,\,\,\,\,\,\,(t,x)\in\R\times\R \notag\\
 (-\Delta_{x})^{\sigma/2}A_0= |\Psi|^2 \notag\\
 \Psi(0,\cdot)=f, \notag
 \end{subnumcases}
 with $\sigma\in(0,1)$, $\alpha=\pm 1$, $1<\gamma\leq 5$, in the spaces $L^2(\R)$ and $H^1(\R)$. \\
\end{abstract}
\maketitle
\section{Introduction}
The nonlinear Schr\"odinger equation
\begin{equation*}
i\partial_t\Psi+\frac{1}{2}\Delta_{x}\Psi = \alpha |\Psi|^{\gamma-1}\Psi,
\end{equation*} 
with $\alpha=\pm 1$,
is one of the universal model to describe the evolution of a wave packet in a weakly nonlinear and dispersive media. In particular, the case $\gamma=3$ occurs to model different physical phenomena: the propagation of waves in optical fibers for $n=1$, the focusing of laser beams for $n=2$, the Bose-Einstain condensation phenomenon for $n=3$, see  \cite{sulemsulem}, \cite{karlsson} and references therein.\\
In the construction of a mathematical model, many physical laws are simplified, so, it is essential to deal with well posed problems: existence of the solution indicates that the model is coherent, uniqueness and stability are related to the problem of approximate the solution with numerical algorithms. The math problem of well-posedness of NLS has been studied for a long time and we can find its history and its current state of the art at the web page \cite{tao} \href{http://www.math.ucla.edu/~tao/Dispersive/}{"Local and global well-posedness for non-linear dispersive and wave equations"} manteined by Colliander, Keel, Staffilani, Takaoka and Tao. In addition, we mention two fondamental monographs specialized in the nonlinear Schr\"odinger equation: Cazenave \cite{cazenave} and Sulem Sulem \cite{sulemsulem}.

 Schr\"odinger-Poisson-Slater system is a nonlinear Schr\"odinger mixed–system which combines
the nonlinear and nonlocal Coulomb interaction, $A_0$, with a local potential nonlinearity
known as the ”Slater exchange term”: 
\begin{subnumcases}
{\,(SPS)}
\partial_t \psi+\frac{1}{2}\Delta \psi =A_0 \psi -C |\psi|^{\gamma-1}\psi\notag\\
-\Delta A_0 = |\psi|^2,\notag
\end{subnumcases}
with $C\geq 0$. Such a model appears in studying of quantum transport in semiconductor devices as a correction to the
Schr\"odinger-Poisson (SP) system ( $C = 0$). For $3D$ well-posedness results of (SPS) in $L^2$ and $H^1$ we mention \cite{bokanowskilopezsoler}. For asymptotic behaviour of $3D$ (SPS) solutions, we mention \cite{sanchez}. One can see a broad literature also about problems concerning the existence and stability
of standing waves for systems like SPS: \cite{fortunato}, \cite{Ruiz}, \cite{nonlinRuiz} , \cite{visciglia}, \cite {georgiev}, \cite{coclitegeorgiev} and references therein.\\ 
Although there are many papers concerning (SPS) and similar systems in 3D, for the case of one space dimension the literature is narrower. The first 1D global results were established for the Maxwell-Schr\"odinger system, which is a generalization of (SP) system that includes the magnetic field. The first result is due to 
Nakamitsu and Tsutsumi \cite{nakamitsutsutsumi} and  uses the Lorentz gauge and high regularity of initial data. Those assumptions imply also the following boundary condition on the electric potential:
\begin{equation*}
A_0(t,x)\rightarrow 0 \,\,\,\,(|x| \rightarrow \infty).
\end{equation*}
Later, Tsutsumi, in \cite{tsutsumi95}, proved that this condition can be relaxed to 
\begin{equation}\label{ultraviolet}
A_0(t,x)\rightarrow c_0 |x| \,\,\, (|x| \rightarrow \infty),
\end{equation}
where,
\begin{equation*}\label{ultraviolet1}
c_0=\frac{1}{2}\int_{\R}|\Psi(0,x)|^2\,dx,
\end{equation*} 
and the initial datum is in $H^1(\R)\cap L^2(\R,|x|\,dx)$.\\ 
Recently, in 1D context, the global well-posedness of the Cauchy problem for the \emph{Hartree equations}
\begin{subnumcases}{ \, }
i\partial_t\Psi+\frac{1}{2}\Delta_{x}\Psi = \lambda A_0\Psi \,\,\,&  $\lambda\in\R$ \notag\\
 (-\Delta_{x})^{\sigma/2}A_0= |\Psi|^2 \,\,\,&  $\sigma\in (0,1)$ \notag
 \end{subnumcases}
with initial data in $H^s(\R),s\geq 0$, was studied in \cite{2007hartreewp} and in \cite{ozawa} (with exchange-correlation correction).\\
 
In this work we study, in one dimensional space, a model like (SPS) with the fractional Poisson equation of Hartree model. Therefore, our attempt is to study a fractional Schr\"odinger-Poisson-Slater (FSPS) system 
\begin{subnumcases}
{ \, }
i\partial_t\Psi+\frac{1}{2}\Delta_{x}\Psi = A_0\Psi +\alpha |\Psi|^{\gamma-1}\Psi  \label{1}\\
 (-\Delta_{x})^{\sigma/2}A_0= |\Psi|^2  \label{2},
 \end{subnumcases}
with $\alpha=\pm 1$ and where $\sigma\in(0,1)$ and $\gamma\in(1,5]$ are chosen such that no boundary condition of type \eqref{ultraviolet} are required.\\
By physical viewpoint, fractional powers of the Laplacian are important in many situations in which one has to consider long-range interaction and anomalous phenomena, see \cite{valdinoci} and references therein. On the other hand, by a mathematical viewpoint, the fractional Poisson equation brings some significant difficulties in the analysis of the well-posedness and allows us to obtain a well-posedness result in one dimension when long range interactions are taken into account.\\

Our goal is to establish existence and uniqueness results about the Cauchy problem (FSPS) with initial data in $L^2(\R)$ and $H^1{(\R)}$. We give a sketch of the plan of the work.\\
At first, following the work of Kato \cite{kato}, we rewrite the Cauchy problem (FSPS) as the integral equation
 \begin{align*}\label{integralform1}
 \Psi(t)= S(t)f-i\int_{0}^{t}S(t-s)& A_0(\Psi(s))\Psi(s)\,ds\\
 &-i\alpha \int_{0}^{t}S(t-s)|\Psi|^{\gamma-1}(s)\Psi(s)\,ds,
 \end{align*}
 where $S(t)$ denotes the Schr\"odinger group $ \eu ^{i\frac{\Delta}{2}t}$ and the electric potential $A_0$ solves the fractional Poisson equation \eqref{2}.\\
We deal with local solvability of the initial value problem in $L^2(\R)$ with standard \emph{contraction argument} obtained by linear techniques (\emph{Strichartz estimates}). The problem is finding at least one admissible pair, $(q,r)$, for which the classical contraction argument works at the same time for the nonlocal and for the local nonlinearity. Indeed, the nonlocal term required to introduce some convolution estimates. We have the following main result:
\begin{thm}\label{11}
Let $f\in L^2(\R)$ and $\gamma\neq 5$. \\
Then, there exists an interval $I_\gamma\subseteq (0,1)$, such that, for all $\sigma\in I_\gamma$, one can find a time $T=T(\|f\|_{L^2})>0$ and a unique wave function $\Psi$, 
$$\Psi\colon[0,T]\times \R\to \C,$$
solution of the Cauchy problem \eqref{1}-\eqref{2}. \\
In addition, we have
\begin{equation*}\label{}
\Psi \in C([0,T],L^2)\cap L^q([0,T],L^r),
\end{equation*}
for any $(q,r)$ admissible pair.
\end{thm}
We note that, if $\gamma=3$, the Theorem \ref{11} holds for $\sigma\in(0,\frac{1}{2}]$.\\
The problem of extending the local solution to all times can be solved thanks to \emph{conservation laws} of $L^2$-norm (charge or mass conservation):
\begin{thm}\label{21}
 Let $f\in L^2(\R)$, $\sigma\in I_\gamma$ and $\gamma\neq 5$. Then, the Cauchy problem \eqref{1}-\eqref{2} has a unique global solution $\Psi \in  C(\R,L^2(\R))\cap L^q(\R,L^r(\R))$, for any $(q,r)$ admissible pair.
 \end{thm}
 The critical case, $\gamma=5$ is more delicate, but the problem lies only in the local nonlinearity. So, we have local well-posedness for large data and global well-posedness for small data:
 \begin{thm}\label{3}
 Let $f\in L^2(\R)$, $\sigma\in I_5$ and $\gamma=5$. There exists a maximal interval $(-T_{min},T_{max})$, $T_{min}=T_{min}(f)$ and $T_{max}=T_{max}(f)$, such that the Cauchy problem \eqref{1}-\eqref{2} has a unique solution
 $$\Psi \in  C([-T_{min},T_{max}],L^2(\R))\cap L^q([-T_{min},T_{max}],L^r(\R)),$$
  for any $(q,r)$ admissible pair.
 \end{thm}
 \begin{thm}\label{4}Let $f\in L^2(\R)$, $\sigma\in I_5$ and $\gamma=5$. There exists a small $\delta$ such that, if $\|f\|_{L^2}\leq \delta$, then the Cauchy problem \eqref{1}-\eqref{2} has a unique global solution $\Psi \in  C(\R,L^2(\R))\cap L^q(\R,L^r(\R))$, for any $(q,r)$ admissible pair.
 \end{thm}
 Next we would like to perform the same previous result with initial data in $H^1(\R)$. The $H^1(\R)$ theory distinguishes the defocusing case ($\alpha=1$) and the focusing case ($\alpha=-1$). In particular the second case is more delicate. We have the following results: 
 \begin{thm}\label{5}
 Let  $f\in H^1(\R)$, $\sigma\in I_\gamma$, $\gamma\neq 5$ and $\alpha=\pm 1$. \\
 Then, there exists a time $T=T(\|f\|_{H^1})>0$, such that one can find a unique wave function $\Psi$, 
 $$
 \Psi\colon[0,T]\times \R\to \C,
 $$
 solution of the Cauchy problem \eqref{1}-\eqref{2} and
 \begin{equation*}\label{}
 \Psi \in C([0,T],H^1(\R))\cap L^q([0,T],W^{1,r}(\R)),
 \end{equation*}
 for any $(q,r)$ admissible pair.
 \end{thm}
 \begin{thm}\label{6}
 Let $f\in H^1(\R)$, $\sigma\in I_\gamma$, $\gamma\neq 5$ and $\alpha=-1$. \\
 Then, the (FSPS) system has a unique global solution $\Psi \in  C(\R,H^1(\R))\cap L^q(\R,W^{1,r}(\R))$, for any $(q,r)$ admissible pair.\\
 Otherwise, if $\gamma=5$, there exists $\delta>0$ such that, if $\|f\|_{L^2}<\delta$ then the Cauchy problem \eqref{1}-\eqref{2} has a unique global solution $\Psi \in  C(\R,H^1(\R))\cap L^q(\R,W^{1,r}(\R))$, for any $(q,r)$ admissible pair.
 \end{thm}
 
 Our
 plan
 in
 this
 paper
 is
 as
 follows. In Section $2$ we introduce some notations and basic fact about LS and NLS: decay estimates, Strichartz estimates, NLS well-posedness results in $L^2$ and $H^1$. The Section $3$ is devoted to well posed problem of (FSPS) system: at first we treat the $L^2$ theory (proof of the Theorems \eqref{11}, \eqref{21}, \eqref{3}, \eqref{4}) and then the $H^1$ theory (proof of the Theorems \eqref{5}, \eqref{6}). In the Section $4$, we establish some decay estimates for the solution of (FSPS): we control the $L^4 L^\infty$-norm of the solution with initial data in $L^2(\R)$ and cubic nonlinearity. Lastly, we get a control estimate for the speed of the oscillation of the solution with initial data in $H^1(\R)$.\\
 
 \textbf{Acknoledgment.} It is a pleasure to acknowledge the interesting conversations about the one dimension Maxwell-Schr\"odinger system \cite{tsutsumi95} with T. Ozawa.\\    
 The author has been supported by  Comenius project "Dynamat" 2010, Universit\`{a} di Pisa and FIRB "Dinamiche Dispersive: Analisi di Fourier e Metodi Variazionali" 2012. 

\section{Preliminaries}
We first introduce some notations.\\
Let $\varphi\in \mathscr S(\R^n)$, a Schwartz function. We define the Fourier transform of $\varphi$ and its inverse as follows:
\begin{align*}
\mathscr F [\varphi](\xi)&= \hat \varphi (\xi)= \frac{1}{(2\pi)^{n/2}}\int_{\R^n}\eu^{-ix\cdot\xi}\varphi(x)\,dx,\\
\mathscr F^{-1} [\varphi](x)&= \check \varphi(x)= \frac{1}{(2\pi)^{n/2}}\int_{\R^n}\eu^{ix\cdot\xi}\varphi(\xi)\,d\xi,
\end{align*}
and then we can extend this operator on tempered distribution $S'(\R^n)$.

The fractional Laplacian $(-\Delta)^{\sigma/2}$ is a pseudo-differential operator defined as: 
\begin{equation}\label{fraclapl}
(-\Delta)^{\sigma/2}A(x)= \mathscr F^{-1}[|\xi|^{\sigma}\hat{A}](x),
\end{equation}
with $A\in S'(\R^n) $
and $0<\sigma<n$. One can see Stein \cite{stein} for a detailed theory on Riesz potentials.

In this work we will consider the Lebesgue spaces $L^p(\R)$, the Sobolev spaces $H^s(\R)$,  and some Bochner spaces, $L^q([0,T],L^r(\R))$, $L^q([0,T],W^{1,r}(\R))$ and $C([0,T],L^2(\R))$. \\
For the Borel-mesaurable functions $g(t,x)\colon [0,T]\times\R\to\C$, $f(x)\colon \R\to\C$, we define the norms of the spaces listed above:
\begin{gather*}
\|f\|_{L^p} =\left(\int_\R |f|^p\,dx \right)^{1/p}, 1\leq p<\infty;\\
\|f\|_{L^\infty} =\supess_\R |f| ;\\
\|f\|_{H^s} = \|\mathscr F^{-1}[ \langle \xi \rangle^s\hat f]\|_{L^2}= \|\langle \xi \rangle^s\hat f\|_{L^2}, \,\,\,\,s\in\R;\\
\|g\|_{L^q([0,T],L^r(\R))} =\left(\int_{0}^{T}\|g(t)\|_{L^r}^{q} \,dt\right)^{1/q};\,\, 1\leq q,r<\infty\\
\|g\|_{L^q([0,T],W^{1,r}(\R))} =\left(\int_{0}^{T}\|g\|_{W^{1,r}}^{q} \,dt\right)^{1/q};\,\, 1\leq q,r<\infty\\
\|g\|_{C([0,T],L^r(\R))} = \sup_{[0,T]} \|g(t,\cdot)\|_{L^r}, \,\, 1\leq r\leq\infty.
\end{gather*}

\subsection{Linear estimates of the free Schr\"odinger equation}
We introduce some basic fact about linear Schr\"odinger equation
\begin{subnumcases}
{(LS)}
i\partial_t\Psi+\frac{1}{2}\Delta\Psi =0 \notag\\
 \Psi(0,\cdot)=f, \notag
 \end{subnumcases}
 If $f\in S(\R^n)$, the Cauchy problem (LS) has a unique solution, $$\Psi(t)=S(t)f,$$ where $$S(t)=\eu^{i\frac{\Delta}{2}t}\colon S(\R^n)\to S(\R^n),$$ is defined by Fourier transform:
 \begin{equation*}
 S(t)f= \mathscr F^{-1}(\eu^{-i\frac{|\xi|^2}{2}t}\hat f).
 \end{equation*}
 By duality we can extend $S(t)$ to $S'(\R^n)$.\\
 In addition, by the proprieties of Fourier transform in $S'(\R)$, we can rewrite the solution $\Psi$ as following
 \begin{equation*}
 \Psi(t)= S(t)f=  \mathscr F^{-1}(\eu^{-i\frac{|\xi|^2}{2}t})* f= \frac{1}{(2\pi i t)^{n/2}}\eu^{i\frac{|\cdot|^2}{2t}}*f.
 \end{equation*}
 We summarize the \emph{time-dispersion} estimates of linear Schr\"odinger in the following Lemma.
\begin{lem}
Let $2 \leq p \leq \infty$, $\frac{1}{p}+\frac{1}{p'}= 1$. There exists $C>0$, such that, for all $f\in L^1 \cap L^2$,
\begin{align*}
\|S(t)f\|_{L^p} & \leq C \frac{1}{t^{n/2-n/p}}\|f\|_{L^{p'}}.
\end{align*}
\end{lem}
\begin{remark}
The Schr\"odinger group, $S(t)$, generates a dispersive effect on initial data, i.e, the initial pulse spreads out after a while because of plane waves with large wave
number travel faster than those with a smaller one.
\end{remark}

Now we summarize the decay estimates for the linear nonhomogeneous Schr\"odinger equation 
\begin{subnumcases}
{\,}
i\partial_t\Psi+\frac{1}{2}\Delta\Psi =F(t,x) \,\,\,\,\,\,(t,x)\in \R\times \R\notag\\
 \Psi(0,\cdot)=f, \notag
 \end{subnumcases}
 in the following Lemma:
 \begin{lem}
 Let $2\leq p\leq \infty $, $f\in L^{p'}$ and $F\in {L^\infty [(0,T),L^{p'}]}$ for $T>0$. Then there exists a constant $C=C(p)>0$ such that 
 \begin{equation*}
 \|\Psi(t)\|_{L^p}\leq t^{-1/2+1/p}\|f\|_{L^{p'}}+C \int_{0}^{t}(t-s)^{-1/2+1/p}\|F(s)\|_{L^{p'}}\,ds,
 \end{equation*} 
 for $t\in (0,T)$.
 \end{lem}
  These dispersive estimates are remarkable but is not quite handy for solving the nonlinear problems. In a perturbative regime we need to space-time estimates. We begin by introducing the notion of admissible pair.
 \begin{definition}
 We say that a pair $(q,r)$, is admissible if
 \begin{equation*}\label{admiss}
 \frac{2}{q}=\frac{n}{2}-\frac{n}{r},
 \end{equation*}
 and 
 \begin{align*}
 2\leq  r \leq \frac{2n}{n-2} \,\,\,&\text{ if } n\geq 3, \\
 2\leq   r <\infty \,\,\,&\text{ if } n=2 ,\\
 2\leq   r \leq \infty \,\,\,&\text{ if } n=1.
 \end{align*}
 \end{definition}
 \begin{remark}
 Scaling argument for Strichartz estimates say us that these restrictions on the pair $(q,r)$ are necessary. The pairs $(2,\frac{2n}{n-2})$, $n\geq 3$, are called \emph{endpoint}.
 \end{remark}
\begin{thm}[Strichartz's estimates]
Let $(q,r)$, $(\tilde q,\tilde r)$ be two Schr\"odinger admissible pairs. Then, the following estimates hold: 
\begin{gather}
\| S(t)f\|_{L^q(\R ,L^r(\R^n))}\leq C \|f\|_{L^2(\R^n)},\label{1str}\\
\| \int_{\R}S^{*}(t)F(t)\,dt\|_{L^2(\R^n)}\leq C \|F\|_{L^{\tilde q'}(\R ,L^{\tilde r'}(\R^n))},\label{2str}\\
\|\int_{0}^{t}S(t-s)F(s)\,ds \|_{L^q_t(\R ,L^r_x(\R))}\leq C \| F \|_{L^{\tilde q'}(\R ,L^{\tilde r'}(\R^n))}\label{3str}.
\end{gather}
With $S^{*}(t)=\eu^{-i\frac{\Delta}{2}t}$ we denote the adjoint of $S(t)= \eu^{i\frac{\Delta}{2}t}$.
\end{thm}
For a complete proof of the Theorem one can see \cite{keeltao}.
\begin{remark}\label{crucial}
The pairs $(q,r)$, $(\tilde q,\tilde r)$ are not related to each other in
the Strichartz's estimates. This turns out to be a crucial fact for the nonlinear applications.
\end{remark}

\subsection{A class of semilinear Schr\"odinger equations}

One of the most important class of nonlinear Schr\"odinger equations are the following:
\begin{equation}\label{nonlschr}
i\partial_t\Psi+\frac{1}{2}\Delta\Psi =\pm |\Psi|^{\gamma-1}\Psi,
\end{equation}
with $\gamma>1$.\\
As we can see in \eqref{nonlschr}, the evolution is a competition between the linear part and the nonlinear one. So we can expect that the evolution has \emph{linearly dominated behavior} or \emph{nonlinearly dominated behavior} or \emph{intermediate behavior}. Nonlinear physics phenomena are characterized by a variety
of complex phenomena; e.g. shock-waves, solitons and instabilities, hence, in a predominantly nonlinear regime we can expect a tricky scenario.\\

So, we are interested in classifying the nonlinearity. Two basic features are crucial: the conservation laws and the natural scale-invariance of the equation.\\
Thanks to the structure of the equation \eqref{nonlschr}, in $H^1(\R)$, the following conservation laws hold:
\begin{itemize}
\item Mass conservation:
\begin{equation*}
\|\Psi(t)\|_{L^2}= \|\Psi(0)\|_{L^2},
\end{equation*}
\item Energy conservation:
\begin{equation*}
E(\Psi(t))= \frac{1}{4}\|\nabla \Psi (t)\|_{L^2}^2\pm\frac{1}{\gamma+1}\|\Psi(t)\|_{L^{\gamma+1}}^{\gamma+1}=E(\Psi(0)),
\end{equation*}
\item Momentum conservation:
\begin{equation*}
\rm{Im} \left( \int \nabla \Psi (t,x)\bar{\Psi}(t,x)\,dx\right)=\rm{Im} \left( \int \nabla \Psi (0,x)\bar{\Psi}(0,x)\,dx\right).
\end{equation*}
\end{itemize}
Using the scale-invariance for \eqref{nonlschr}
\begin{equation}\label{scalingsolution}
\Psi_{\lambda}(t,x)=\lambda^{2/(1-\gamma)}\Psi(\frac{t}{\lambda^2},\frac{x}{\lambda}),
\end{equation}
for $\lambda>0$, we can classify the conservation laws as \emph{subcritical}, \emph{critical} (scale-invariant), or \emph{supercritical}. \\
In particular, in one dimension, using $L^2$-conservation (similarly for $H^s$ conservation), we have
\begin{equation}\label{rescaling}
\|\Psi_{\lambda}(t,\cdot)\|_{L^2}=\lambda^\frac{5-\gamma}{2(1-\gamma)}\|\Psi(t,\cdot)\|_{L^2}.
\end{equation}
We can give the following definition:
\begin{definition}
Let $\gamma >1$, we say that
\begin{itemize}
\item $\gamma$ is $L^2$-subcritical if $1<\gamma< 5$,
\item $\gamma$ is $L^2$-critical if $\gamma= 5$,
\item $\gamma$ is $L^2$-supercritical if $\gamma>5$.
\end{itemize}
\end{definition}
The rescaling relation \eqref{rescaling}, maight be interpreted as following: in subcritical case, the norm of the initial data can be made small while the interval of time is made longer; in supercritical case, the norm grows as the time interval gets longer; finally, in the critical case, the norm is invariant while the interval of time is made longer or shorter: this looks like a limit situation for well-posedness results.

Another most important distinction is whether the equation is \emph{focusing} ($\alpha=-1$) or \emph{defocusing} ($\alpha=1$). We can not make an exact distinction, but, broadly, in a defocusing case, the nonlinearity has the same sign as the linear component, thus, the dispersive effects of the linear equation are amplified. On the contrary, in the focusing case the dispersive effects can be attenuated, halted (stationary or travelling waves can occur) or even reversed (blow up of solution in finite time can occur).
	
Except to $1$-dim cubic NLS, the equations are not completly integrable. We are interestested in the fundamental question of \emph{well-posedness} that is often closely intertwined with the quantitative estimates (\emph{a priori estimates}).

For some literature on local existence results in the subcritical case, one can see \cite{ginibrevelo79}, \cite{kato}, \cite{tsutsumi87} and \cite{cazenave}. For local existence in the critical case, one can see \cite{cazenave90}, \cite{cazenave}. Finally, for the global critical case one can see \cite{ginibrevelo}, \cite{cazenave}. \\
Now we state some of the results that will come in handy later.
\begin{thm}[$L^2$ 
 well-posedness] (Cazenave \cite{cazenave}, Section 4.6)
Let $f\in L^2(\R)$. The following statements hold: 
\begin{itemize}
\item Let $1<\gamma<5$ and $\alpha=\pm 1$. Then there exists a unique global solution $\Psi\in C(\R,L^2)\cap L^q (\R,L^r(\R))$;
\item Let $\gamma=5$ and $\alpha=\pm 1$. Then there exists $\delta>0$, quite small, such that, if $\|f\|_{L^2}\leq \delta$ we have a unique global solution $\Psi\in C(\R,L^2)\cap L^q (\R,L^r(\R))$.
\end{itemize}
\end{thm} 
\begin{thm}[$H^1$ well-posedness](Cazenave \cite{cazenave}, Section 4.4)
Let $f\in H^1(\R)$. The following statements hold: 
\begin{itemize}
\item Let $1<\gamma <5$ and $\alpha=\pm 1$. Then there exists a unique global solution $\Psi\in \C(\R,H^2)\cap L^q (\R,W^{1,r}(\R))$; 
\item Let $\gamma=5$ and $\alpha=-1$. Then there exists $\delta>0$, quite small, such that, if $\|f\|_{L^2}\leq \delta$ we have a unique global solution $\Psi\in C(\R,H^1)\cap L^q(\R, W^{1,r}(\R))$. 
\end{itemize}
\end{thm}

In $H^1$-theory, about global existence results for small data when $\gamma=5$, looking for the sharp mass $\delta$ which allows to obtain global well-posedness is interlinked with the problem of the best constant in the Gagliardo-Nirenberg interpolation estimates.
	\section {Well-posedness of the Cauchy problem}
We now turn to the system (FSPS). \\
We say that an initial-value problem for a partial differential equation is well-posed in $H^s$ if:
\begin{itemize}
\item there exists a time interval, $[0,T]$, in which the problem in fact has an $H^s$ solution,
\item the solution is unique,
\item the solution depends continuosly on the initial data.
\end{itemize}
 
We seek for an handy formulation of the Cauchy problem $(FSPS)$ to begin with.
\begin{rem}\label{propsps} (Stein \cite{stein}, Section 5.1)
Let $0<\sigma <1$. Then  
\begin{equation}\label{trasformatafraz}
\mathscr F (|x|^{-\sigma})(\xi)=C(\sigma) |\xi|^{-(1-\sigma)},
\end{equation}
where 
$$
C(\sigma)= \sqrt{\pi}2^{1-\sigma}\frac{\Gamma(\frac{1-\sigma}{2})}{\Gamma(\frac{\sigma}{2})}.
$$
The equation \eqref{trasformatafraz} is understood in the sense of the tempered distributions.
\end{rem}
\begin{lem}[Hardy-Littlewood-Sobolev Inequality]\label{hls}
Let $f\colon \R^n\to\C$ a mesaurable function, $1<r<p<\infty$ and $f\in L^r(\R^n)$. Let $0<\beta <n$. Then there exists $C=C_{p,\beta,n}$ a positive constant such that
\begin{equation}\label{hlsi}
\left\|\frac{1}{|y|^{\beta}}*f\right\|_{L^p}\leq C \|f\|_{L^r},
\end{equation}
where, 
$\frac{1}{p}=\frac{\beta}{n}+\frac{1}{r}-1$.
\end{lem}
For a proof of this Lemma see Stein \cite{stein}. 
\begin{prop}
Let $0<\sigma<1$, $T>0$ and let $(q,r)$ be an admissible pair with $2<r<\frac{2}{\sigma}$. Suppose that $\Psi\colon[0,T]\times\R\to\C$ a known function and $\Psi\in  L^q([0,T],L^r)$.\\
Then, there exists a unique electric potential $A_0$, 
$$A_0\colon[0,T]\times \R\to \R,$$
\begin{equation}\label{convolform}
A_0(t,x)=C(\sigma)[|\cdot|^{-(1-\sigma)}*|\Psi|^2](t,x),
\end{equation}
\begin{equation}
A_0 \in L^{q/2}([0,T],L^{r/(2-r\sigma)}(\R)),
\end{equation}
solution of the fractional Poisson equation 
\begin{equation}\label{fractionalpoisson}
(-\Delta)^{-\sigma/2}A_0=|\Psi|^2.
\end{equation}
\end{prop}
\begin{proof}
By the \eqref{fraclapl} and by \eqref{fractionalpoisson} we have that
\begin{equation}
 \hat A_0(\xi)=|\xi|^{-\sigma} (|\Psi|^2 \hat)(\xi).
\end{equation}
Thanks to the Lemma \ref{propsps} and passing under Fourier antitransform, we have that 
\begin{equation*}
A_0(t,x)=C(\sigma)\left[ |\cdot|^{-(1-\sigma)}*|\Psi|^2(t,\cdot)\right]  \,(x).
\end{equation*}
Hence the equality \eqref{convolform} has been proved.\\
By the hypothesis on $\Psi$ and by the H\"older inequality we have that $|\Psi|^2\in L^{q/2}([0,T],L^{r/2}_x(\R))$. So, the Lemma \ref{hlsi} tells us that $A_0 \in L^{q/2}([0,T],L^{\frac{r}{2-r\sigma}}_x(\R))$. \\
The unicity of the electric potential $A_0$ is guaranteed by the unicity of the wave function $\Psi$ and by injectivity of Fourier transform.
\end{proof}
Now we bring us back to study the following Cauchy problem:
\begin{subnumcases}
{\,}
i\partial_t\Psi+\frac{1}{2}\Delta\Psi =C(\sigma)[|\cdot|^{-(1-\sigma)}*|\Psi|^2]\Psi +\alpha |\Psi|^{\gamma-1}\Psi \,  \label{sps1}\\
\Psi(0,\cdot)=f. \label{sps3} 
 \end{subnumcases}
 First of all we specify which kind of solutions we are searching for. We give the following definition:
 \begin{definition}
 Let $X_0$ be a Banach space, $f\in X_0$ and $T>0$. We consider the map
 \begin{equation}\label{mappa}
 \mathscr H[\Psi](t)=S(t)f-iC(\sigma)\int_{0}^{t}S(t-s)[|\cdot|^{-(1-\sigma)}*|\Psi|^2]\Psi(s)\,ds-i\alpha\int_{0}^{t}S(t-s) |\Psi|^{\gamma-1}\Psi(s) \,ds,
 \end{equation}
 with $t\in[0,T]$.\\
 We say that $\Psi\in C([0,T],X_0)$ is a local solution of \eqref{sps1}-\eqref{sps3} if $\Psi$ is a fixed point of the map $\mathscr H$, i.e. $\Psi=\mathscr H (\Psi)$.
 \end{definition}
 \subsection{Well-posed problems with initial data in $L^2$}
We start with $L^2$-theory. By means of the contraction theory, we can prove, in subcritical case, a local existence result (Theorem \ref{lex}). Actually, in this case , thanks to the mass conservation, we can extend the local result to a global one (Corollary 4.9). On the other hand, in the $L^2$-critical case we can prove a local result with large data (Theorem \ref{lecld}) and a global result with small data (Theorem \ref{gecsd}). \\
 Notice that the $L^2$-theory does not see the difference between the defocusing case ($\alpha=1$) and the focusing case ($\alpha=-1$).
\begin{thm}[local existence $L^2$-subcritical]\label{lex}
 
Let $1<\gamma<5$, $\alpha=\pm1$ and $f\in L^2(\R)$. \\
Then, there exists an interval $I_\gamma \subseteq (0,1)$, such that, if $\sigma\in I_\gamma$, one can find a time $T=T(\|f\|_{L^2})>0$ and a unique wave function $\Psi$, 
$$\Psi\colon[0,T]\times \R \to \C ,$$
solution of the problem \eqref{sps1}-\eqref{sps3}. \\
In addition, we have
\begin{equation}\label{regolonda}
\Psi \in C([0,T],L^2(\R))\cap L^q([0,T],L^r(\R)),
\end{equation}
for any $(q,r)$ admissible pair.
\end{thm}
\begin{proof}
As we mentioned before, we are going to proof the theorem by means of a contraction argument. Hence, we have to introduce a suitable Banach space, $X_{0}$, and then we have to prove that $\mathscr H\colon X_0\to X_0$, defined in \eqref{mappa}, is a contraction.\\
Unlike classical NLS, we have also a nonlocal term
\begin{equation*}\label{nonlocalterm}
C(\sigma)[|\cdot|^{-(1-\sigma)}*|\Psi|^2]\Psi,
\end{equation*}
that will bring necessary modification.\\

Let $f\in L^2$ be the initial data. \\
Let $T=T(\|f\|_{L^2})>0$ and $M=M(\|f\|_{L^2})$ be two positive constant which will be defined later and, let $(q,r)$ be an admissible pair
\begin{equation}
\frac{1}{q}=\frac{1}{4}-\frac{1}{2r}.
\end{equation} \\
We will denote the spaces $ L^\infty([0,T],L^2(\R))$ and $L^q([0,T],L^r(\R))$ as $L^\infty L^2$ and $L^q L^r$ respectively to simplify the notation. \\
Let $X_0$ be the Banach space defined as follows:
\begin{equation*}
X_0=\left\lbrace  \Psi \in L^\infty L^2\cap L^q L^r |  \,\Psi(0)=f, \,\|\Psi\|_{X_0}= \|\Psi\|_{L^\infty L^2}+ \|\Psi\|_{L^q L^r}\leq M\right\rbrace. 
\end{equation*}
We will prove that $\mathscr H$ is a contraction on $X_0$.\\

\textbf{Step 1.} (Looking for a working admissible pair)\\
To apply classical estimates that work also on nonlocal term, we have to make some considerations.\\
In order that the Strichartz estimates might give back the desired norm, we have to choose the admissible pair, $(q,r)$, such that we can be able to construct the pairs $(\tilde{q}',\tilde{r}')$ and $(\tilde{q}'_1,\tilde{r}'_1)$ as follows.

At first we consider the nonlocal term. By the Hardy-Littlewood-Sobolev hypothesis and by the condition of Schr\"odinger admissibility on $(\tilde{q}',\tilde{r}')$ we have: 
	\begin{subnumcases}
	{\,}
	0<\sigma<1, \,\,\,\notag\\
	2<r<\frac{2}{\sigma}\notag\\
	\frac{3}{1+\sigma} \leq r \leq \frac{6}{1+2\sigma} \notag\\
	 \frac{1}{\tilde{r}'}=\frac{3}{r}-\sigma\notag\\
	  \frac{1}{\tilde{q}'}= \frac{1}{x}+\frac{3}{q}= \frac{1+\sigma}{2}+\frac{3}{q}\notag\\
	  \frac{1}{\tilde{q}'}=\frac{5}{4}-\frac{1}{2\tilde{r}'}.\notag
	\end{subnumcases}
On the other hand, the local nonlinearity brings the following conditions:
	\begin{subnumcases}
	{\,}
	\gamma\leq r \leq 2\gamma, \notag\\
	 \frac{1}{\tilde{r}'_1}=\frac{\gamma}{r}\notag\\
	  \frac{1}{\tilde{q}'_1}= \frac{1}{x}+\frac{\gamma}{q}= \frac{5-\gamma}{2}+\frac{3}{q}\notag\\
	  \frac{1}{\tilde{q}'_1}=\frac{5}{4}-\frac{1}{2\tilde{r}'_1}\notag
	\end{subnumcases}

Hence, we would of course want the interval 
 \begin{equation*}
    I_{\sigma,\gamma}= (2,\frac{2}{\sigma})\cap[\frac{3}{1+\sigma}, \frac{6}{1+2\sigma}]\cap  [\gamma,2\gamma]
\end{equation*}
to be not empty for all $\sigma\in(0,1)$ and for all $\gamma\in(0,5)$. It is not possible. Indeed, if $\gamma=3$, we have that the set $I_{\sigma,\gamma}\neq \emptyset$ iff $\sigma\in (0,1/2]$.\\
We can represent the relations between $\gamma$ and $\sigma$ such that $I_{\sigma,\gamma}$ is not empty. In particular, in the picture below, the coloured region rapresents the set of the pairs $(\sigma,\gamma)$ for wich our proof works. 

\begin{figure}[h!hhh]
\includegraphics[width=0.8\textwidth]{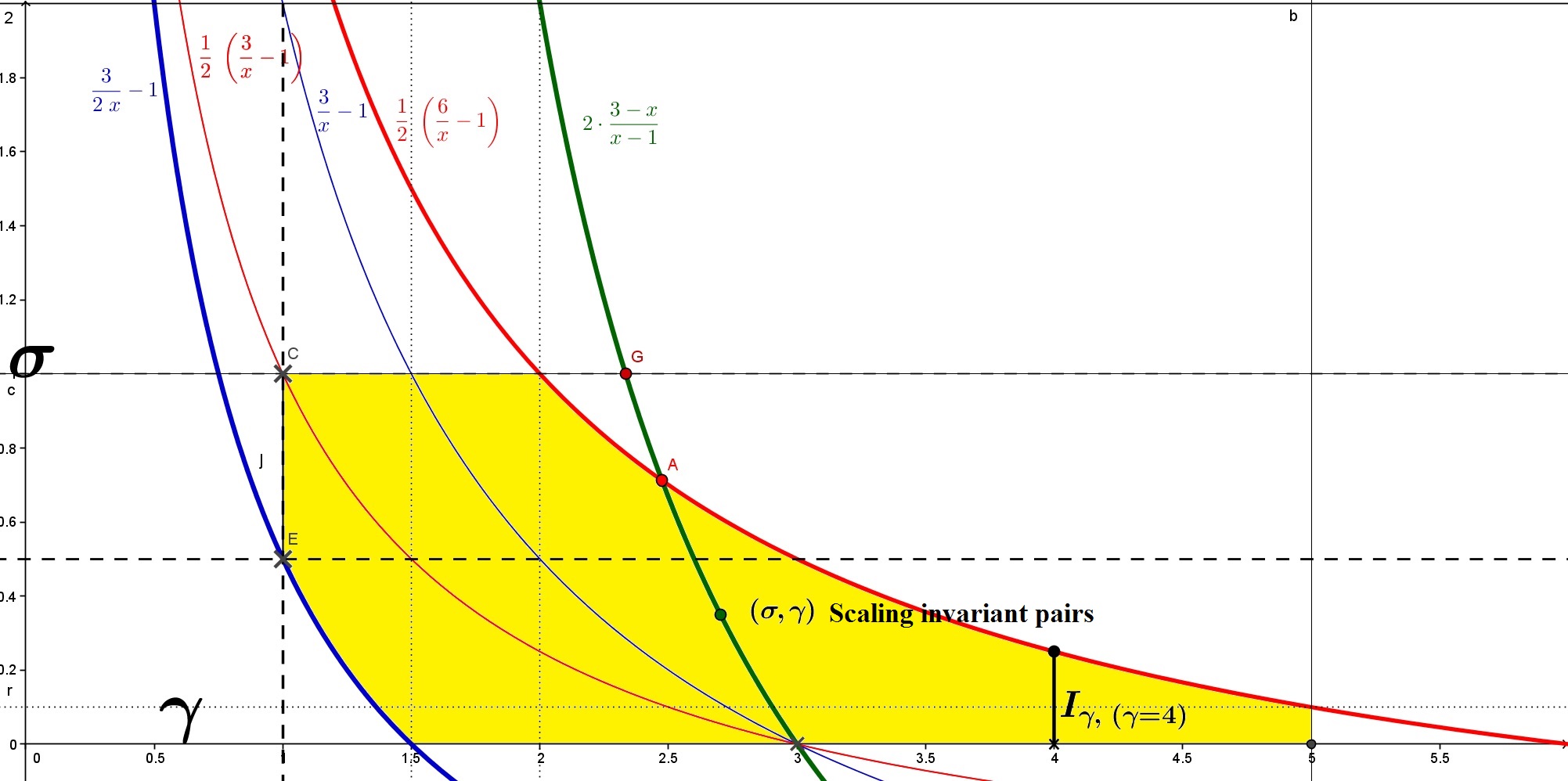}
\end{figure}

Hence, if we fix $\gamma\in (1,5)$, there exists an interval $I_\gamma$, such that, if $\sigma\in I_\gamma$, then we can construct the admissible pairs $(\tilde{q}',\tilde{r}')$ and $(\tilde{q}'_1,\tilde{r}'_1)$ as specified above.
Without loss of generality, in the following, we will consider $\sigma$ small enough, in particular, $\sigma\in(0,1/10]$ and $3/2\leq\gamma<5$.
    
\textbf{Step 2.}\,($\mathscr H$ is a contraction on $X_0$)\\
Proof that $X_0$ is mapped into itself by $\mathscr H$:
\begin{align}
\|\mathscr H \Psi\|_{X_0}&\leq \| S(t)f \|_{X_0} +C\|\int_{0}^{t}S(t-s) [|\cdot|^{-(1-\sigma)}*|\Psi|^2](s,\cdot)\Psi(s)\,ds\|_{X_0}\label{pallainpalla}\\
 \,\,\,\,\,\,&\,\,\,\,\,\,\,\,\,\,\,\,\,\,\,\,\,\,\,\,\,\,\,\,\,\,\,\,\,\,\,\,\,\,\,\,\,\,\,\,\,\,\,\,\,\,\,\,\,\,\,\,\,\,\,\,\,\,\,\,\,\, +\|\int_{0}^{t}S(t-s) |\Psi|^{\gamma-1}\Psi(s)\,ds \|_{X_0}\notag\\
&\leq C \|f\|_{L^2} +C\left\|[|\cdot|^{-(1-\sigma)}*|\Psi|^2]  \Psi \right\|_{L^{\tilde q'}L^{\tilde r'}}+C\||\Psi|^{\gamma-1}\Psi\|_{L^{\tilde q'_1}L^{\tilde r'_1}} \notag\\
&\leq C \|f\|_{L^2}+C\left\| \||\cdot|^{-(1-\sigma)}*|\Psi|^2 \|_{L^{r/(2-r\sigma)}}\| \Psi\|_{L^r} \right\|_{L^{\tilde q'}}+C\left\| \| \Psi\|_{L^r}^\gamma \right\|_{L^{\tilde q'_1}}\notag\\
&\leq C \|f\|_{L^2}+C \left\| \| \Psi\|^3_{L^r} \right\|_{L^{\tilde q'}}+C\left\| \| \Psi\|_{L^r}^\gamma \right\|_{L^{\tilde q'_1}}\notag\\
&\leq C \|f\|_{L^2}+C T^{\frac{1}{2}+\frac{\sigma}{2}}\left\|  \Psi\right\|^3_{L^{q}L^r}+CT^{\frac{5-\gamma}{2}}\left\|  \Psi\right\|^\gamma_{L^{q}L^r},\notag
\end{align}
where we have used the Strichartz estimates, the H\"older inequality and the Lemma \ref{hlsi}. Note that $C$ depends on the constants involved in \eqref{1str}-\eqref{3str} and \eqref{hlsi}.\\
We put $M=3C\|f\|_{L^2}$. If $T=T(\|f\|_{L^2})$ is quite small, then we get
\begin{equation}\label{finita}
\| \mathscr H \Psi\|_{X_0}\leq 3C\|f\|_{L^2}=M.
\end{equation}
Now we want to prove that $\mathscr H$ is a contraction.\\
We have that the following estimates hold: 
\begin{align*}
\left| (|\cdot|^{-(1-\sigma)}*|\Psi_1|^2)  \Psi_1-(|\cdot|^{-(1-\sigma)}*|\Psi_2|^2)  \Psi_2 \right|\leq& (|\cdot|^{-(1-\sigma)}*|\Psi_1|^2)\left|\Psi_1-\Psi_2 \right| \\ &+(|\cdot|^{-(1-\sigma)}*\left[ (|\Psi_1|-|\Psi_2|)(|\Psi_1|+|\Psi_2|)\right])|\Psi_2| 
\end{align*}
and
\begin{equation*}
\left|\Psi_{1}|\Psi_{1}|^{\beta-1} -\Psi_{2}|\Psi_{2}|^{\beta-1}\right| \leq C |\Psi_1-\Psi_2|(|\Psi_1|^{\beta-1}+|\Psi_2|^{\beta-1}),
\end{equation*}
for $\beta>1$. As in \eqref{pallainpalla}, we can prove that $\mathscr H$ is a contraction. \\
Indeed, let $\Psi_1,\Psi_2\in X_0$, we have that
  \begin{align*}
  \| \mathscr H (\Psi_1)-\mathscr H (\Psi_2)\|_{X_0}&\leq  C T ^{1/2+\sigma/2}\|\Psi_1-\Psi_2\|_{X_0}(\|\Psi_1\|^2_{X_0}+\|\Psi_2\|^2_{X_0})\\
    &\,\,\,\,\,+CT^{\frac{5-\gamma}{2}}\|\Psi_1-\Psi_2\|_{X_0}(\|\Psi_1\|^{\gamma-1}_{X_0}+\|\Psi_2\|^{\gamma-1}_{X_0})\\
  & \leq 2CT^{1/2+\sigma/2}(3C\|f\|_{L^2})^2\|\Psi_1-\Psi_2\|_{X_0}\\
  & \,\,\,\,+ 2CT^{\frac{5-\gamma}{2}}(3C\|f\|_{L^2})^{\gamma-1}\|\Psi_1-\Psi_2\|_{X_0}.
  \end{align*}
  Choosing $T=T(\|f\|_{L^2})$ small enough, we get 
  \begin{equation}\label{getcontraction}
   \| \mathscr H (\Psi_1)-\mathscr H (\Psi_2)\|_{X_0}\leq \frac{1}{2}\|\Psi_1-\Psi_2\|_{X_0}.
  \end{equation} 
  The Banach fixed point theorem guarantees the existence and uniqueness of $\Psi\in X_{0}$, such that $\mathscr H (\Psi)=\Psi$.\\
   Hence, there exists a unique wave function $\Psi$, solution of Cauchy problem \eqref{sps1}-\eqref{sps3} and its continuity in time is immediate a posteriori by the \eqref{mappa}. Actually, we have had an additional regularity information: $\Psi\in L^\gamma([0,T],L^\rho(\R))$ for any $(\gamma,\rho)$ admissible pair. It follows by Strichartz estimates \eqref{1str}-\eqref{3str} and by the \eqref{finita}: 
    \begin{equation*}
    \| \Psi\|_{L^\gamma([0,T],L^\rho(\R))}\leq 3C\|f\|_{L^2}.
    \end{equation*}
    
\textbf{Step 3.} (Continuos dependence by initial data)\\
Now we deduce the continuous dependence from initial data to complete the local well-posedness of the Cauchy problem \eqref{sps1}-\eqref{sps3}.\\
Let $f,g\in L^2(\R)$. Let $\Psi(f)$ and $\Psi(g)$ be the solutions of the Cauchy problem \eqref{sps1} with initial data $f$ and $g$ respectively. Writing the solutions in the integral form \eqref{mappa}, a computation like in \eqref{getcontraction} tells us
\begin{equation}
\| \Psi(f)-\Psi(g)\|_{X_0}\leq C\|f-g\|_{L^2}.
\end{equation}
\end{proof}
\begin{rem}
Note that if $\sigma=\frac{2(3-\gamma)}{\gamma-1}$ the problem \eqref{sps1} is scale invariant. That is, if $\Psi $ solves \eqref{sps1}, then $\Psi_{\lambda}$, defined as in \eqref{scalingsolution}, is still a solution. So, in subcritical case, we may expect well-posedness beyond yellow region, at least on the path $\sigma=\frac{2(3-\gamma)}{\gamma-1}$ (green path in the picture above), with $\sigma\in(0,1)$. This may suggest looking for other skills to proof local well posed results.
\end{rem}
Once local existence is established, some natural  issues are the following. What is the existence time of the solution? Can we extend the local solution to global one? Can blow-up phenomena occur? We will try to answer them.\\

At first, we state a conservation law.
\begin{lem}[Conservation mass]\label{conservationmass}
Let $\Psi \in C([0,T],L^2(\R)) $ be a local solution of the Cauchy problem \eqref{sps1}-\eqref{sps3}. Then 
\begin{equation}\label{massa}
\|\Psi (t,\cdot)\|_{L^2}= \|f\|_{L^2},
\end{equation}
for all times $t\in [0,T]$.
\end{lem}
\begin{proof}
We assume $\Psi \in C([0,T],H^1(\R)) $ and we multiply by $\bar{\Psi}$ the equation \eqref{sps1}. The \eqref{massa} follows by integration by parts. Density arguments give us the general statement (for a detailed proof see \cite{cazenave} Section 4.6).
\end{proof}

Our goal is to try to extend the solution to all times. At first we define the \emph{maximal solution} using the uniqueness for small time.
\begin{definition}
Let $f\in L^2$. Let
\begin{align*}
T_{max}&= \sup \{T>0, \text{such that \eqref{sps1} has a solution in } [0,T]\},\\
T_{min}&= \sup \{T>0, \text{such that \eqref{sps1} has a solution in } [-T,0]\}.
\end{align*}
The uniqueness for small time allows us to define the \emph{maximal solution} $$\Psi \in  C([-T_{min},T_{max}],L^2).$$
\end{definition}

\begin{prop}[Blow-up alternative]
Let $T_{max}<\infty$ (respectively , if $T_{min}<\infty$ ), then, under the hypothesis of the Theorem \ref{lex} we have
  \begin{equation}\label{blowup}
\lim_{t \nearrow  T_{max}}\|\Psi(t,\cdot)\|_{L^2(\R)}=\infty \,\,\,(\lim_{t \searrow T_{min}}\|\Psi(t,\cdot)\|_{L^2(\R)}=\infty).
\end{equation}
\end{prop}
\begin{proof}
Let $T_{max}<\infty$. Assume that there exist $M<\infty$ and a sequence $t_n \nearrow T_{max}$ such that $\|\Psi (t_n)\|_{L^2(\R)}\leq M$. \\

We consider $ k \in \mathbf N $, such that $t_k+T(M)>T_{max}$, where $[0,T(M)]$ denotes the maximal existence interval of a solution with initial data of $L^2$-norm equals to $M$. \\
By Theorem \ref{lex} and starting from $f=\Psi(t_k)$, we can extend $\Psi$ up to $t_k+T(M)$, which contradicts maximality.
\end{proof}

\begin{cor}\label{gser}
Let $f\in L^2$, $1<\gamma<5$ and $\sigma\in I_\gamma$. The Cauchy problem \eqref{sps1}-\eqref{sps3} has a unique global solution $\Psi \in  C(\R,L^2(\R))\cap L^q(\R,L^r(\R))$, for any $(q,r)$ admissible pair.
\end{cor}
\begin{proof}
By the mass conservation and by the blow-up alternative we have that the local solution is actually global.
\end{proof}

\begin{thm}[local existence $L^2$-critical with large data]\label{lecld}
Let $f\in L^2$, $\gamma=5$ and $\sigma\in(0,1/10]$. There exists a maximal interval $(-T_{min},T_{max})$, $T_{min}=T_{min}(f)$ and $T_{max}=T_{max}(f)$, such that the Cauchy problem \eqref{sps1}-\eqref{sps3} has a unique solution ,$\Psi$, such that
$$\Psi \in  C([-T_{min},T_{max}],L^2(\R))\cap L^q([-T_{min},T_{max}],L^r(\R)),$$
 for any $(q,r)$ admissible pair.
\end{thm}
\begin{proof}
We proceed in a similar way to how we did in the subcritical case and we use the same notations of the Theorem \ref{lex}. The difficulty, in the critical case, lies in the local nonlinear term $|\Psi|^{4}\Psi$. Indeed, let $T>0$, by Strichartz estimates we have that 
\begin{equation*}
   \| \mathscr H (\Psi)\|_{L^{q}([0,T],L^r)}\leq \|S(t)f\|_{L^{q}([0,T],L^r)}+C T^{\frac{1+\sigma}{2}}\left\|  \Psi\right\|^{3}_{L^{q}([0,T],L^r)} +C\left\|  \Psi\right\|^{5}_{L^{q}([0,T],L^r)}.
  \end{equation*} 
  By the Strichartz estimate \eqref{1str} and by absolute continuity of the Lebesgue integral, if $T$ is suitably small, we have that $\|S(t)f\|_{L^{q}([0,T],L^r)}<\delta$, for some small $\delta$ depending on $f$ and on the constant in the Strichartz estimates. Hence,
  \begin{equation}\label{palla}
  \| \mathscr H (\Psi)\|_{L^{q}([0,T],L^r)}\leq \delta+\delta +C\left\|  \Psi\right\|_{L^{q}([0,T],L^r)}^{5}.
  \end{equation}
  We choose $M=3\delta$, for small $\delta$, i.e. for time interval sufficiently small. So, we have a unique fixed point $\Psi\in L^{q}([0,T],L^r) $, which locally solves the Cauchy problem. \\
  In order to conclude the proof we will prove that $\Psi$ is actually also $L^{\infty}([0,T],L^2)$.\\
  For the \eqref{palla}, we have that $\|\Psi\|_{L^{q}([0,T],L^r)}<\infty$. By Strichartz estimates we have
  \begin{equation*}
  \|\Psi\|_{L^\infty([0,T],L^2)}\leq \|f\|_{L^2}+C T^{\frac{1+\sigma}{2}}\left\|  \Psi\right\|^{3}_{L^{q}([0,T],L^r)}+C\left\|  \Psi\right\|_{L^{q}([0,T],L^r)}^{5}.
  \end{equation*}
\end{proof}

\begin{thm}[global existence $L^2$-critical small data]\label{gecsd}
Let $f\in L^2(\R)$, $\gamma=5$ and $\sigma\in(0,1/10]$. There exists a small $\delta>0$ such that, if $\|f\|_{L^2}\leq \delta$, then the Cauchy problem \eqref{sps1}-\eqref{sps3} has a unique global solution $\Psi \in  C(\R,L^2(\R))\cap L^q(\R,L^r(\R))$, for any $(q,r)$ admissible pair.
\end{thm}
\begin{proof}
By the condition $\gamma=5$ follows that
\begin{equation*}
\|\mathscr H \Psi \|_{X_0}\leq \delta +C T^{\frac{1+\sigma}{2}}\left\|  \Psi\right\|^{3}_{X_0}+\|\Psi\|_{X_0}^5.
\end{equation*}
So, if we choose $T=T(\delta)$ and $M= 3C \delta$, for $\delta$ sufficiently small we have that the ball with radius $M$ in $X_0$ is mapped into itself by $\mathscr H$. Similarly we prove that $\mathscr H $ is a contraction. As in subcritical case we deduce first the local well-posedness and then the global result.
\end{proof}

\subsection{Well-posed problems with initial data in $H^1$}
Here we want to perform the same previous results about well-posedness in the space $H^1(\R)$. In this case, with regard to global well-posed problem, the defocusing and focusing case are situations more different. In the defocusing case, thanks to contraction arguments and energy conservation, we have the same results of $L^2$-theory. Therefore, we focus our attenction on focusing case that, already in subcritical case, is quite complicated.\\

At first we construct the local solution in $C([0,T],H^1(\R))$, in subcritical focusing and defocusing case, with a fixed point argument.
\begin{thm}[local existence $H^1$-subcritical]\label{lth1}
Let $f\in H^1(\R)$, $1<\gamma<5$ and $\alpha=\pm1$.\\
Then, there exists an interval $I_\gamma\subseteq (0,1)$, such that, for all $\sigma\in I_\gamma$ one can find a time $T=T(\|f\|_{H^1})>0$ and a unique wave function $\Psi$, 
$$\Psi\colon[0,T]\times \R\to \C,$$
solution of the problem \eqref{sps1}-\eqref{sps3}. \\
In addition, we have
\begin{equation}\label{regolonda}
\Psi \in C([0,T],H^1(\R))\cap L^q([0,T],W^{1,r}(\R)),
\end{equation}
for any $(q,r)$ admissible pair.
\end{thm}
\begin{proof}
The construction of the local $H^1$-solution is entirely similar to the construction of the local solution in $L^2$-theory. We give only a sketch of the proof. 
We introduce the Banach space  $X_0$ defined as following:
\begin{equation*}
X_0=\left\lbrace  \Psi \in L^\infty H^1\cap L^q W^{1,r}, \Psi(0)=f, \|\Psi\|_{X_0}= \|\Psi\|_{L^\infty H^1}+ \|\Psi\|_{L^q W^{1,r}}\leq M\right\rbrace. 
\end{equation*}
Let $\Psi_1,\Psi_2\in X_0$. We have that the inequalities
\begin{equation*}
\left|\nabla[\Psi_{1}|\Psi_{1}|^{\gamma-1}] \right| \leq 2C |\Psi_1|^{\gamma-1}|\nabla\Psi_1|,
\end{equation*}
and
\begin{equation*}
\left|\nabla[ (|\cdot|^{-(1-\sigma)}*|\Psi_{1}|^2)\Psi_{1}] \right| \leq 2 [|\cdot|^{-(1-\sigma)}*(\Psi_1\nabla\Psi_1)]|\Psi_1|+(|\cdot|^{-(1-\sigma)}*|\Psi_1|^2)|\nabla\Psi_1|
\end{equation*}
come true almost everywhere.\\ 
Moreover, we have similar estimates for the difference of the gradients.

 Thanks to these inequalities we construct the solution as a fixed point of the contraction map $\mathscr H$ in $X_0$ with the same arguments of the Theorem \ref{lex}.
 \end{proof}
\begin{definition}[Energy]
Let $\Psi \in C([0,T],H^1)$, with $T>0$. We define the \textit{energy} of the system \eqref{sps1}-\eqref{sps3} as follows:
\begin{equation}\label{energy}
E(t)= \frac{1}{4}\|\nabla \Psi(t)\|_{L^2}^2+\frac{1}{4} \int_{\R}A_0|\Psi|^2(t,x)\,dx +\frac{\alpha}{\gamma+1}\int_{\R}|\Psi|^{\gamma+1}(t,x)\,dx,
\end{equation}
for any $t\in [0,T]$. \\
Sobolev embedding ensure that the energy is well-defined.
\begin{lem}[Conservation energy]
Let $\Psi \in C([0,T],H^1(\R)) $ be a local solution of the Cauchy problem \eqref{sps1}-\eqref{sps3}. Then 
\begin{equation}\label{energycons}
E(t)=E(0),
\end{equation}
for all times $t\in [0,T]$.
\end{lem}
\begin{proof}
 We assume $\Psi \in C^{1}([0,T],H^2)$. Multiplying by $\partial_t\bar{\Psi}$ the equation \eqref{sps1}, similarly to Lemma \ref{conservationmass}, we deduce the conservation energy. Thanks to continuous dependence on initial data guaranteed by the local well-posedness (Theorem \ref{lth1}), density arguments prove that the quantity \eqref{energy} is a constant during the evolution of the system. 
 \end{proof}
\end{definition}

\begin{prop}[Blow-up alternative]\label{bualternative}
Let $T_{max}<\infty$ (respectively , if $T_{min}<\infty$ ), then, under the hypothesis of the Theorem \ref{lth1} we have
  \begin{equation}\label{blowup}
\lim_{t \nearrow  T_{max}}\|\Psi(t,\cdot)\|_{H^1(\R)}=\infty \,\,\,(\lim_{t \searrow T_{min}}\|\Psi(t,\cdot)\|_{H^1(\R)}=\infty).
\end{equation}
\end{prop}
\begin{rem}
In defocusing case, $\alpha=1$, we have that the energy, $E(t)=E(0)$, is a positive constant. \\
So we have that 
\begin{equation}\label{h1finita}
\|\nabla \Psi\|_{L^2}^2 \leq E(0).
\end{equation}
Hence, thanks to \eqref{h1finita}, Lemma \ref{conservationmass} and by Proposition \ref{bualternative} we can extend the $H^1$-local solution to global one.

\end{rem}

\begin{cor}[$H^1$ global existence - defocusing case]\label{gser}
Let $f\in H^1$, $1<\gamma<5$, $\sigma\in I_\gamma$ and $\alpha=+1$. The Cauchy problem \eqref{sps1}-\eqref{sps3} has a unique global solution $\Psi \in  C(\R,H^1(\R))\cap L^q(\R,W^{1,r}(\R))$, for any  $(q,r)$ admissible pair.
\end{cor}

Now we give more attenction to focusing case.
We have the following result.
\begin{thm}\label{h1focus}
Let $f\in H^1(\R)$ and $\alpha=-1$. \\
If $1<\gamma <5$ and $\sigma\in I_\gamma $, then, the Cauchy problem \eqref{sps1}-\eqref{sps3} has a unique global solution $\Psi \in  C(\R,H^1(\R))\cap L^q(\R,W^{1,r}(\R))$, for any $(q,r)$ admissible pair.\\
Otherwise, if $\gamma=5$ and $\sigma\in I_5$, then, there exists $\delta>0$ such that, if $\|f\|_{L^2}<\delta$ then the Cauchy problem \eqref{sps1}-\eqref{sps3} has a unique global solution $\Psi \in  C(\R,H^1(\R))\cap L^q(\R,W^{1,r}(\R))$, for any  $(q,r)$ admissible pair.
\end{thm}
\begin{proof}
The local solution is found by means of a point fix argument in the proof of the Theorem \ref{lth1}. To conclude that the solution actually is global, since $L^2$-norm is conserved, it is sufficient to prove that the norm $\|\nabla \Psi(t,\cdot)\|_{L^2}$ does not blow up.\\
By the Gagliardo-Nirenberg inequality, there exists $C_{GN}>0$ (the sharp constant) such that
\begin{equation}\label{gagliardon}
\|f\|_{\beta+1}^{\beta+1}\leq C_{GN} \|\nabla f\|_{L^2}^{(\beta-1)/2}\|f\|_{L^2}^{(\beta+3)/2},
\end{equation}
for $1\leq \beta <\infty$.\\
So, choosing $\beta=\gamma$, we obtain that 
\begin{align*}
E(0)=E(t)\geq&  \frac{1}{4}\|\nabla \Psi(t)\|_{L^2}^2+\frac{1}{4}\int_{\R}A_0(t)|\Psi(t)|^2\,dx-\frac{C_{GN} }{\gamma+1} \|\nabla \Psi(t)\|_{L^2}^{(\gamma-1)/2}\|\Psi(t)\|_{L^2}^{(\gamma+3)/2}\\
\geq& \frac{1}{4}\|\nabla \Psi(t)\|_{L^2}^2\left( 1-4\frac{C_{GN}}{\gamma+1}\|\nabla \Psi(t)\|_{L^2}^{(\gamma-5)/2}\|\Psi(t)\|_{L^2}^{(\gamma+3)/2} \right). 
\end{align*}
Since the mass is costant, $\|\Psi(t)\|_{L^2}=\|f\|_{L^2}$, if $1<\gamma <5$ ($\sigma\in I_\gamma$), we have that the $H^1$-norm cannot blow up. Indeed, if $\|\nabla \Psi\|_{L^2}$ was large we would control it with the energy:  
\begin{equation}\label{wpH1}
\|\nabla \Psi(t)\|_{L^2}\leq C E(0),
\end{equation}
for some $C>0$. \\
So, the proof in subcritical case is complete.

In the critical case, $\gamma= 5$ ($\sigma\in I_5$), we have that 
\begin{align*}
E(0)\geq& \frac{1}{4}\|\nabla \Psi(t)\|_{L^2}^2\left( 1-4\frac{C_{GN} }{\gamma+1}\|\Psi(t)\|_{L^2}^{4} \right)\\
=& \frac{1}{4}\|\nabla \Psi(t)\|_{L^2}^2\left( 1-\frac{\|f\|_{L^2}^{4}  }{\left( \sqrt[4]{\frac{2}{3}}\sqrt{\frac{\pi}{2}}\right)^4}\right) . 
\end{align*}
As a consequence, if we choose initial data with $L^2$-norm suitably small, $\|f\|_{L^2}<\delta$, with $\delta = \sqrt[4]{\frac{2}{3}}\sqrt{\frac{\pi}{2}} $, we have the \eqref{wpH1}, and the proof is complete. 
\end{proof}
\begin{rem}
The sharp constant for the Gagliardo-Nirenberg inequality in one dimensional setting was derived by Nagy in 1941; Weinstein in 1983 solved the problem for higher dimensions.
\end{rem}
\begin{rem}
The physical meaning of the Theorem \ref{h1focus} is that for waves propagating in a weakly focusing medium ($1<\gamma<5$), the potential term in the energy, $ E_{pot}=\frac{-1}{\gamma+1}\|\Psi\|^{\gamma+1}_{L^{\gamma+1}}$, is dominated by interaction term, $E_{int}=\frac{1}{4} \int_{\R}A_0|\Psi|^2(t,x)\,dx $, and by kinetic term, $ E_{kin}=\frac{1}{4}\|\nabla \Psi(t)\|_{L^2}^2$, in according to Gagliardo-Nirenberg inequality. \\
When $\gamma=5$, the potential energy and the kinetic one, seems to balance out, so, global results, at least in the case of focusing nonlinear Schr\"odinger, are not guaranteed.
\end{rem}

\section{$L^4-L^\infty$ estimates of the solution of cubic (FSPS)}
In this section we will use $L^p-L^q$ Gronwall's inequalities to establish a decay estimate for the solution of the (FSPS) system. Compare with Cazenave, we have the following result.
\begin{lem}\label{2gron}
Let $1\leq q<p\leq \infty$, $1\leq \rho <\infty$ with $\frac{1}{\rho}=\frac{1}{q}-\frac{1}{p}$, $C_1>0$ and $0<T\leq \infty$. We consider $a\in L^{\rho}(0,T)$ and $v$ a function that satisfies the following inequality:
\begin{equation}\label{ipotesigro}
\|v\|_{L^p(0,t)}\leq C_1+\|av\|_{L^q(0,t)},
\end{equation}
for all $t\in (0,T)$. \\
Then 
\begin{equation}\label{gammagronwall}
\|v\|_{L^p(0,t)}\leq 2C_1 \Gamma(2+2^{\rho}\|a\|^{\rho}_{L^{\rho}(0,t)}),
\end{equation}
for all $t\in (0,T]$.
\end{lem}
\begin{proof}
Suppose $\|a\|_ {L^{\rho}(0,T)}\geq1/2$. We can partition the interval $(0,T)$ into $n$ parts, with $n\geq 2$, such that  $(\tau_{k})_{\{0\leq k\leq n\}}$ is an increasing sequence of time, $\tau_{0}=0$, $\tau_{n}=T$ and 
\begin{equation}
\|a\|_{L^{\rho}(\tau_{k-1},\tau_{k})}=\frac{1}{2},\,\, \text{$1\leq k\leq n-1$}; \, \text{ } \, \|a\|_{L^{\rho}(\tau_{n-1},\tau_{n})}\leq \frac{1}{2}.
\end{equation}
So, we have that
\begin{align*}
\int_{0}^{T}|a|^{\rho}\,ds=&  \int_{0}^{\tau_1}|a|^{\rho}\,ds+\dots +\int_{\tau_{n-1}}^{T}|a|^{\rho}\,ds\\
\leq& \frac{1}{2^{\rho}}+\dots + \frac{1}{2^{\rho}}= \frac{n}{2^{\rho}}.
\end{align*}
We put $n= [2^\rho \|a\|_{L^\rho(0,T)}^\rho]+1$.\\
Set $a_0=0$ and $a_{k}= \|v\|_{L^p(0,\tau_k)}$.
By the \eqref{ipotesigro} and by the H\"older inequality, we have 
that 
\begin{align}
a_{k+1}=& \|v\|_{L^p(0,\tau_{k+1})}\leq C_1+ \|av\|_{L^q(0,\tau_{k})}+\|av\|_{L^q(\tau_{k},\tau_{k+1})}\notag\\
\leq& C_1 + \|a\|_{L^\rho(0,\tau_{k})}\|v\|_{L^p(0,\tau_{k})}+\|a\|_{L^\rho(\tau_{k},\tau_{k+1})}\|v\|_{L^p(\tau_{k},\tau_{k+1})}\notag\\
\leq& C_1+\frac{k}{2} a_{k} + \frac{1}{2} a_{k+1}\notag
\end{align}
So we have that 
\begin{equation}
a_{k+1}\leq 2C_1+ k a_{k},
\end{equation}
hence
\begin{align*}
a_{k+1}\leq& 2C_1\left( 1+k+ k(k-1)+ k(k-1)(k-2)+ \dots +  k(k-1)(k-2)\dots 2\cdot 1 \right)\\
 \leq& 2C_1 (k+1)k!\leq  2C_1( k+1)!,
\end{align*}
Let $t\in[\tau_{k},\tau_{k+1}]$. Then 
\begin{equation}
\|a\|_{L^\rho(0,t)}\geq \|a\|_{L^\rho(0,\tau_{k})}= \frac{k^{1/\rho}}{2},
\end{equation}
and so
\begin{equation}
k\leq 2^\rho \|a\|^{\rho}_{L^\rho(0,t)}.
\end{equation}
A simple working gets the thesis. Actually, we have
\begin{equation}
\|v\|_{L^\rho(0,t)}\leq a_{k+1}\leq 2C_1( k+1)!= 2C_1\Gamma(k+2) \leq 2C_1 \Gamma (2+2^\rho\|a\|^\rho_{L^\rho(0,t)}).
\end{equation}
Else if $\|a\|_ {L^{\rho}(0,T)}<1/2$, trivially, we have that
\begin{equation*}
\|v\|_{L^{p}(0,T)}\leq 2C_1.
\end{equation*}
\end{proof}
In different context the next Lemma is useful. 
\begin{lem}\label{1gron}
Let $C_2>0$, $ 1 \leq q < p \leq \infty$ and $v(t) \in C([0,+\infty))$, $a(t) \in L^\infty_{loc}((0,+\infty))$ are positive functions that satisfy the following inequalities
\begin{equation}\label{eq.a2.1}
    \|v\|_{L^p(0,1)}^q \leq C_2,
\end{equation}
\begin{equation}\label{eq.a2.2}
    \|v\|_{L^p(0,t)}^q \leq C_2 +   \int_1^t a(\tau) v(\tau)^q \,d\tau  \,\,\,\text{\,\, $t>1$,}
\end{equation}
then 
\begin{equation}\label{eq.a2.5}
  \|v\|_{L^p(0,t)}  \leq \left(\frac{p}{p-q} \right)^{1/p}C_2^{1/q} \exp\left( \frac{1}{p} \left(\frac{p}{q} \right)^{p/(p-q)}\|a\|_{L^{\frac{p}{p-q}}(1,t)}^{\frac{p}{p-q}} \right),
\end{equation}
with the obvious modifications for $p=\infty$.
\end{lem}
\begin{proof}
By the \eqref{eq.a2.2} we have that 
\begin{equation}\label{eq.a2.2a}
  t > 1 \Longrightarrow  \|v\|_{L^p(0,t)}^q \leq C_2 +  \int_1^t a(\tau) v^q(\tau) d\tau .
\end{equation}
Set
$$ \varphi(t) = C_2+\int_1^t a(\tau) v^q(\tau) d\tau.$$
Then we have the relation
$$ v^q (t) = \frac{\varphi^\prime(t)}{a(t)}, $$
so we have
\begin{equation}\label{eq.a2.2b}
   \left\|\frac{\varphi^\prime}{a}\right\|_{L^{p/q}(1,t)} \leq \varphi(t).
\end{equation}
To simplify the notation, we set $\alpha=p/q$ and $\alpha^\prime= p/(p-q)$.\\
We can use the inequality \eqref{eq.a2.2b} to derive the estimates
\begin{alignat}{3} \nonumber
     \varphi^\alpha(t) =& \varphi^\alpha(1) + \alpha\int_1^t \varphi^\prime(\tau) \varphi^{\alpha-1}(\tau) d \tau  \\
   \nonumber
   =& C_2^\alpha +\alpha\int_1^t \frac{\varphi^\prime(\tau)}{a(\tau)} \varphi^{\alpha-1}(\tau)a(\tau) d \tau \\
   \nonumber
   \leq&  C_2^\alpha + \alpha \left\|\frac{\varphi^\prime}{a} \right\|_{L^{\alpha}(1,t)} \left\|a\varphi^{\alpha-1} \right\|_{L^{\alpha'}(1,t)}\\
   \nonumber
   \leq&  C_2^\alpha+\frac{\varphi^\alpha(t) }{\alpha} + \frac{\alpha^{\alpha'}\left\|a\varphi^{\alpha-1} \right\|_{L^{\alpha'}(1,t)}^{\alpha' }}{\alpha'}.
\end{alignat}
We can rewrite the inequality above as 
\begin{equation*}
\frac{\varphi^\alpha(t)}{\alpha'}\leq C_2^\alpha+ \frac{\alpha^{\alpha'}}{\alpha'}\int_1^t  a^{\alpha'}(\tau) \varphi^\alpha (\tau) \,d\tau,
\end{equation*}
so, we are in position to apply classical Gronwall's inequality and derive that 
\begin{equation*}
\varphi^\alpha(t)\leq \alpha' C_2^{\alpha} \exp\left(\alpha^{\alpha'} \int_1^t  a^{\alpha'}(\tau) \,d\tau \right).
\end{equation*}
Rise to $\frac{1}{\alpha}$ and by the \eqref{eq.a2.2a} we obtain the \eqref{eq.a2.5}.
\end{proof}
Actually, if we have $a\in\L^\infty(0,T)$, for large time $T$, the inequality \eqref{eq.a2.5} is better than \eqref{gammagronwall}.\\

Now we consider the particular case $\gamma=3$ and $\sigma \in (0,1/2]$.

\begin{prop}[$L^4(0,T)L^\infty$ no blow-up result]
Let $T>0$ and let $\Psi\in C([0,T];L^2)$ be the solution of \eqref{sps1}-\eqref{sps3} with initial data $f\in L^2$. Then we have that 
\begin{equation}
\|\Psi\|_{L^4(0,t)L^\infty}<\infty
\end{equation}
for all $t\in (0,T)$.
\end{prop}
\begin{proof}
Let $T>0$ and $t\in(0,T)$. 
By the Strichartz estimates combined with H\"older inequality and Hardy-Littlewood-Sobolev inequality we get:
\begin{align*}
\|\Psi\|_{L^4(0,t)L^\infty } &\leq C\|f\|_{L^2}+ \|A_0\Psi+\alpha |\Psi|^2\Psi\|_{L^{4/3}(0,t)L^1}\\
&\leq C\|f\|_{L^2}+\|A_0\Psi\|_{L^{4/3}(0,t)L^1}+\| |\Psi|^2\Psi\|_{L^{4/3}(0,t)L^1}\\
&\leq C\|f\|_{L^2}+C\|f\|_{L^2}^{2(1+\sigma)}\|\langle s \rangle^{\frac{3}{2}\frac{\sigma}{1-2\sigma}} \Psi \|^{1-2\sigma}_{L^{4/3}(0,t)L^\infty }+C\|f\|_{L^2}^2 \|\Psi\|_{L^{4/3}(0,t)L^\infty}.
\end{align*}
Putting $C_{f}=\max \left\lbrace  C\|f\|_{L^2},C\|f\|_{L^2} ^{2(1+\sigma)},C\|f\|_{L^2}^2 \right\rbrace $, we have that:
\begin{equation*}
\|\Psi\|_{L^4(0,t)L^\infty } \leq C_f+C_f \|\langle s \rangle^{\frac{3}{2}\frac{\sigma}{1-2\sigma}} \Psi \|_{L^{4/3}(0,t)L^\infty }.
\end{equation*}
By the $L^p-L^q$ Gronwall's inequality \eqref{gammagronwall} (similarly with \eqref{eq.a2.5}), we get the thesis:
\begin{equation}
\|\Psi\|_{L^4(0,t)L^\infty_{x} }\leq 2C_f \Gamma \left( 2+ (2C_f)^{\frac{2}{3}\frac{1-2\sigma}{\sigma}}t^2 \right)  .
\end{equation} 
\end{proof}

We conclude with a remark on the \emph{speed of the oscillation} of the solution of (FSPS).
\begin{lem}
Let $T>0$, $ \rho, S\colon (0,T)\times \R \to\R$ and let $\Psi=\rho(t,x)\eu^{iS(t,x)}$ be the solution of \eqref{sps1}-\eqref{sps3}, with initial data $f\in H^1$ and $\alpha=-1$. Define the functions $h\colon \R\times \R \to \R$ and $\theta\colon \R \to \R$ in according to \cite{cazenavebl}:
\begin{gather}
h(t)=\partial_tS\\
\theta(t)=\int_{\R}|\Psi(t,x)|^2h(t,x)\,dx,
\end{gather}
with $t\in (0,T)$. Then, the following statements on the speed of the oscillation, $\theta(t)$, hold:
\begin{itemize}
\item[1)] If $1<\gamma<5$ and $\sigma\in I_\gamma$ then $\theta(t)$ cannot blow up for all $t\in(0,T)$;
\item[2)] elseif $\gamma=5$ and $\sigma\in I_5$, then there exists $\delta >0$ such that, if $\|f\|_2<\delta$, $\theta(t)$ cannot blow up for all $t\in(0,T)$.
\end{itemize}
\end{lem}
\begin{proof}
Let $1<\gamma<5$ and $\sigma\in I_\gamma$. Suppose $\|f\|_{L^2}=1$.  
A simple computation shows that
\begin{equation}\label{veloscill1}
h(t)= \frac{\Im (\bar{\Psi}\partial_t\Psi)}{|\Psi|^{2}}.
\end{equation}
Multlipying by $\bar{\Psi}$ the equation \eqref{sps1}, integrating by parts and by the \eqref{veloscill1} we have that 
\begin{equation*}
\theta(t)=\int_{\R} h(t)|\Psi|^2\,dx= -2E(0)-\frac{1}{2}\int A_0 |\Psi|^2\,dx -\alpha\frac{\gamma}{\gamma+1}\|\Psi\|_{L^{\gamma+1}}^{\gamma+1}.
\end{equation*}
%
Thanks to Gagliardo Nirenberg inequality \eqref{gagliardon} and by finiteness of the energy we have that $\|\nabla\Psi\|_{L^2}$ cannot blow up. 
Hence, $\|\Psi\|_{\gamma+1}$ cannot blow up. So, we can conclude that $\int A_0\Psi\,dx$ does not blow up. \\
In particular, the following inequality holds
\begin{align*}\label{veloscilla}
|\theta(t)|\leq & 2|E(0)|+(\frac{1}{2}\|f\|_2+\frac{C_{GN}^2}{2}\|\nabla \Psi\|_{L^2}^{1-\frac{1}{p'}}\|f\|_2^{1+\frac{1}{p'}})+\frac{\gamma C_{GN}}{\gamma+1}\|\nabla\Psi\|_{L^2}^{\frac{\gamma-1}{2}} \\
= & 2|E(0)|+(\frac{1}{2}+C\|\nabla \Psi\|_{L^2}^{1-\frac{1}{p'}})+C\|\nabla\Psi\|_{L^2}^{\frac{\gamma-1}{2}},
\end{align*}
where $1\leq p<\frac{1}{1-\sigma}$, and $\frac{1}{p}+\frac{1}{p'}=1$.
Hence, the speed of the oscillations (on average) cannot blow up if the initial energy is finite.\\
Note that in defocusing case $\alpha=1$, we get easily
\begin{equation*}
|\theta(t)|\leq C|E(0)|,
\end{equation*}
with $C$ a positive constant.\\
If $\gamma=5$ and $\sigma\in I_5$, we need the smallness of the initial mass $\|f\|_2$ (see the proof of the Theorem \ref{h1focus}) to get the same conclusion.\\
\begin{rem}
The blow-up as a consequence of rotational proprieties of the solution concerning the nonlinear Schr\"odinger problem is studied in \cite{cazenavewebl} and \cite{cazenavebl}.
\end{rem}

\end{proof}

\end{document}